\documentclass{amsart}

\usepackage{amsmath}
\usepackage{amssymb}
\usepackage{verbatim}
\usepackage{dsfont}
\usepackage{enumerate}
\usepackage[all]{xy}
\usepackage[mathscr]{eucal}

\usepackage{amsthm}
\usepackage{stmaryrd}
\usepackage{amscd}
\usepackage{amsfonts}
\usepackage{latexsym}

\usepackage{amscd}

\usepackage{graphicx}

\newtheorem{theorem}{Theorem}[subsection]
\newtheorem{lemma}[theorem]{Lemma}
\newtheorem{proposition}[theorem]{Proposition}
\newtheorem{corollary}[theorem]{Corollary} 
\theoremstyle{definition}  
\newtheorem{definition}[theorem]{Definition}

\newtheorem{remark}[theorem]{Remark}

\def\uRep{\underline{\operatorname{Re}}\!\operatorname{p}} 

\newcommand{\Rep}{\operatorname{Rep}}

\newcommand{\id}{\text{id}}

\newcommand{\Rem}{\operatorname{Rem}}
\newcommand{\Add}{\operatorname{Add}}
\newcommand{\emp}{\varnothing}
\newcommand{\biemp}{{(\varnothing,\varnothing)}}

\newcommand{\I}{{\cat I}}
\newcommand{\C}{{\cat C}}

\newcommand{\Z}{\mathbb{Z}}
\newcommand{\K}{\Bbbk}

\newcommand{\cat}{\mathcal}

\newcommand{\black}{\bullet}
\newcommand{\white}{\circ}
\newcommand{\arup}[1]{\stackrel{#1}{\longrightarrow}}

\newcommand{\down}{\vee}
\newcommand{\up}{\wedge}
\newcommand{\cross}{{\Large\text{$\times$}}}
\newcommand{\bigo}{\bigcirc}

\begin{document}

\title{Ideals in Deligne's  category $\uRep(GL_\delta)$}

\author{Jonathan Comes}
\address{J.C.: Department of Mathematics,
University of Oregon, Eugene, OR 97403, USA}
\email{jcomes@uoregon.edu}

\begin{abstract} We give a classification of ideals in $\uRep(GL_\delta)$ for arbitrary $\delta$.

\end{abstract}

\date{\today}

\maketitle  

\setcounter{tocdepth}{2}

\section{Introduction}

\subsection{}\label{uniprop}
Let $\K$ denote a field of characteristic zero, $V$ a $d$-dimensional vector space over $\K$, and $GL_d=GL(V)$ the corresponding general linear group.  The usual tensor product of vector spaces gives $\Rep(GL_d)$, the category of finite dimensional representations of $GL_d$, the structure of a $\K$-linear symmetric monoidal category.     
More generally, given integers $m,n\geq0$ the standard tensor product of super vector spaces gives $\Rep(GL(m|n))$, the category of representations of the general linear supergroup $GL(m|n)$, the structure of a $\K$-linear symmetric monoidal category and the natural representation has categorical dimension (i.e.~super dimension) $m-n$.
Given an arbitrary $\delta\in\K$, Deligne has defined a $\K$-linear  symmetric monoidal category $\uRep(GL_\delta)$ containing a ``natural" object of categorical dimension $\delta$, which is in some precise sense  universal among all such categories \cite[Proposition 10.3]{Del07}.  In particular, $\Rep(GL(m|n))$ along with the natural representation of $GL(m|n)$ prescribe a tensor functor $F_{m|n}:\uRep(GL_\delta)\to\Rep(GL(m|n))$ whenever $\delta=m-n$.  Let $\I(m|n)$ denote the collection of objects sent to zero by $F_{m|n}$.  $\I(m|n)$ is an example of an ``ideal" in $\uRep(GL_\delta)$.  More generally, we have 
\begin{definition}\label{ideal} A collection of objects $\I$ in a braided monoidal category $\C$ is called an \emph{ideal} of $\C$ if the following conditions are satisfied:
\begin{enumerate}
\item[(i)] $X\otimes Y\in\I$ whenever $X\in\C$ and $Y\in\I$.
\item[(ii)] If $X\in\C$, $Y\in\I$ and there exist $\alpha:X\to Y$, $\beta:Y\to X$ such that $\beta\circ\alpha=\id_X$, then $X\in\I$.
\end{enumerate}
\end{definition}
\noindent We will call an ideal in $\C$ \emph{nontrivial} if it contains a nonzero object, and \emph{proper} if it does not contain all objects in $\C$.  The main result of this paper is a classification of ideals of $\uRep(GL_\delta)$.  More precisely, we show  $\{\I(m|n)~|~m,n\in\Z_{\geq0}, m-n=\delta\}$ is a complete set of pairwise distinct nontrivial proper ideals in $\uRep(GL_\delta)$ (see \S\ref{main result}).

\subsection{} A related notion is that of a \emph{tensor ideal}, i.e.~a collection of morphisms which is closed under arbitrary compositions and tensor products (see, for instance \cite[\S2.4]{CK} for a precise definition).  A classification of tensor ideals in $\uRep(GL_\delta)$ would be a nice result, and the main result of this paper can be viewed as a step towards that classification.  For more details on the relationship between ideals and tensor ideals we refer the reader to  \cite[\S2.5]{CK}.

On the other hand, ideals in the sense of Definition \ref{ideal} are interesting in their own right.  For instance, ideals will sometimes admit modified trace and dimension functions in the sense of \cite{GKP}.   In \cite{CK} it was shown that if $\delta\in\Z_{\geq0}$ then the unique nontrivial proper ideal in Deligne's $\uRep(S_\delta)$  admits a modified trace.  It would be interesting to determine which ideals in $\uRep(GL_\delta)$ admit modified traces.  

\subsection{}  The proof of the main result of this paper relies heavily on results from \cite{CW}.  The following items concerning the category $\uRep(GL_\delta)$ can be found in \emph{loc.~cit.}:

\medskip

I. An explicit construction of $\uRep(GL_\delta)$ using walled Brauer diagrams.

II. A classification of indecomposable objects by bipartitions.

III. A procedure for decomposing tensor products of indecomposable objects.

IV. A complete description of the indecomposable objects in the ideal $\I(m|n)$.

\medskip

\noindent  Although we do not need the explicit construction of $\uRep(GL_\delta)$ in this paper, items II, III, and IV above play a crucial role in our proof of the classification of ideals.  We  focus on a special class of indecomposable objects in $\uRep(GL_\delta)$ which are labelled by so-called almost $(m|n)$-cross bipartitions (see \S\ref{almost}).  Roughly speaking, to prove the main result we use item III above to show that every nontrivial proper ideal in $\uRep(GL_\delta)$ is, in some sense, generated by an indecomposable object labelled by an almost $(m|n)$-cross bipartition.  Then we use item IV above to conclude that all nontrivial proper  ideals must be of the form $\I(m|n)$. 

The decomposition of tensor products in $\uRep(GL_\delta)$ is governed by the combinatorial properties of Brundan and Stroppel's weight diagrams introduced in [BS1-5].  Accordingly, these weight diagrams also play a significant role in this paper.

\subsection{}  The paper is organized as follows.  In \S\ref{bwt} we introduce bipartitions and weight diagrams, the combinatorial objects in this paper.  We prove several basic facts concerning these combinatorial objects, paying special attention to the almost $(m|n)$-cross bipartitions.  In \S\ref{decompo} we recall the classification of indecomposable objects in $\uRep(GL_\delta)$, as well as some useful formulae for decomposing tensor products in $\uRep(GL_\delta)$ found in \cite{CW}.  The remainder of \S\ref{decompo} is devoted to proving some technical results on decomposing tensor products.  Finally, in \S\ref{mainresult} we prove our classification of ideals in $\uRep(GL_\delta)$.

\section{Bipartitions and weight diagrams}\label{bwt}

In this section we lay out the combinatorial tools used in this paper, namely bipartitions and the weight diagrams of Brundan and Stroppel.  Since it takes little space to do so, we provide proofs to most of the properties we will need concerning  these combinatorial objects.  However, we note that many these observations can be found (at least implicitly) in [BS1-5] or even in \cite{CW}.  We will follow the notation in \cite{CW}.

\subsection{Bipartitions}   A \emph{partition} is an tuple of nonnegative integers $\alpha=(\alpha_1, \alpha_2,\ldots)$ 
 such that $\alpha_i\geq\alpha_{i+1}$ for all $i>0$, and $\alpha_i=0$ for all but finitely many $i$.  
 Let $\cat{P}$ denote the set of all partitions and let $\varnothing=(0,0,\ldots)\in\cat{P}$.  
 We write $|\alpha|=\sum_{i>0}\alpha_i$ for the \emph{size} of $\alpha$.  Given $\alpha,\beta\in\cat{P}$, we say \emph{$\alpha$ is contained in $\beta$} and write $\alpha\subset\beta$ if $\alpha_i\leq\beta_i$ for all $i>0$.    
We identify a partition $\alpha$ with its \emph{Young diagram} which consists of left-aligned rows of boxes, with $\alpha_i$ boxes in the $i$th row (reading from top to bottom).  For example, 
$$\includegraphics{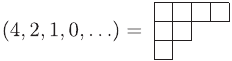}.$$ 

Elements of $\cat{P}\times\cat{P}$ are called \emph{bipartitions}.  Given a bipartition $\lambda$, we let $\lambda^\black$ and $\lambda^\white$ denote the partitions such that $\lambda=(\lambda^\black, \lambda^\white)$.    We will write $|\lambda|=|\lambda^\black|+|\lambda^\white|$ for the \emph{size}\footnote{Note that this differs from the notation in \cite{CW} where $|\lambda|$ denotes $(|\lambda^\black|,|\lambda^\white|)$.} of $\lambda$.  Given bipartitions $\lambda$ and $\mu$, we say \emph{$\lambda$ is contained in $\mu$} and write $\lambda\subset\mu$ if $\lambda^\black\subset\mu^\black$ and $\lambda^\white\subset\mu^\white$.   
Given a bipartition $\lambda=(\lambda^\black,\lambda^\white)$, let $\Add^\black(\lambda)$ (resp.~$\Add^\white(\lambda)$) denote the set of bipartitions obtained from $\lambda$ by adding a box to $\lambda^\black$ (resp.~$\lambda^\white$).  Similarly, let $\Rem^\black(\lambda)$ (resp.~$\Rem^\white(\lambda)$) denote the set of bipartitions obtained from $\lambda$ by removing a box from $\lambda^\black$ (resp.~$\lambda^\white$).   

\subsection{Weight diagrams}  Given $\delta\in\K$ and a bipartition $\lambda$, set $$\begin{array}{rl}
I_\up(\lambda) & =\{\lambda^\black_1, \lambda^\black_2-1, \lambda^\black_3-2, \ldots\},\\
I_\down(\lambda,\delta) & =\{1-\delta-\lambda^\white_1, 2-\delta-\lambda^\white_2, 3-\delta-\lambda^\white_3, \ldots\}.
\end{array}$$  The \emph{weight diagram of $\lambda$ with parameter $\delta$}, denoted $x_\lambda(\delta)$, is the diagram obtained by labeling the integer vertices on the number line according to the following rule:  label the $i$th vertex by $$\left\{\begin{array}{cl}
\bigo & \text{ if }i\not\in I_\up(\lambda)\cup I_\down(\lambda,\delta),\\
\up & \text{ if }i\in I_\up(\lambda)\setminus I_\down(\lambda,\delta),\\
\down & \text{ if }i\in I_\down(\lambda,\delta)\setminus I_\up(\lambda),\\
\cross & \text{ if }i\in I_\up(\lambda)\cap I_\down(\lambda,\delta).
\end{array}\right.$$  
For example, $$\includegraphics{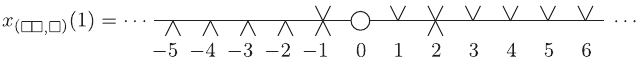},$$
and 
\begin{equation}\label{emptywt}\includegraphics{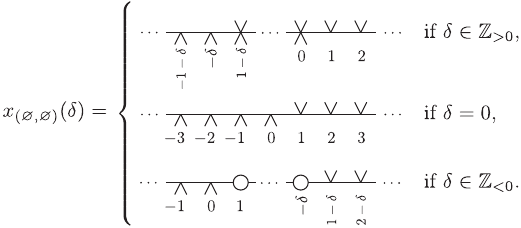}\end{equation}

It will be important for us to understand how adding/removing a box to a bipartition affects its weight diagram.  We record this in the following proposition leaving the proof as a straightforward exercise for the reader.

\begin{proposition}\label{addremwt}  Suppose $\lambda$ and $\mu$ are bipartitions and $\delta\in\Z$.  

(i) $\mu\in\Rem^\black(\lambda)$ (equivalently $\lambda\in\Add^\black(\mu)$) if and only if the labels on $x_\lambda(\delta)$ and $x_\mu(\delta)$ are all the same except an adjacent pair of vertices which are labelled as in one of the following four cases:  $$\begin{array}{|c|c|c|c|c|}\hline
x_\mu(\delta)  & \cross\bigo & \up\bigo &  \cross\down & \up\down\\\hline
x_\lambda(\delta) & \down\up & \bigo\up  & \down\cross & \bigo\cross\\\hline
\end{array}$$

(ii) $\mu\in\Rem^\white(\lambda)$ (equivalently $\lambda\in\Add^\white(\mu)$) if and only if the labels on $x_\lambda(\delta)$ and $x_\mu(\delta)$ are all the same except an adjacent pair of vertices which are labelled as in one of the following four cases:  $$\begin{array}{|c|c|c|c|c|}\hline
x_\mu(\delta) & \bigo\cross & \bigo\down  & \up\cross & \up\down\\\hline
x_\lambda(\delta) & \down\up & \down\bigo  & \cross\up  & \cross\bigo\\\hline
\end{array}$$
\end{proposition}

Using Proposition \ref{addremwt} we can prove the following statement which allows us to recover $\delta$ from any given weight diagram whenever  $\delta$ is an integer.

\begin{corollary}\label{xminuso} For any integer $\delta$ and bipartition $\lambda$, $\delta$ is equal to the number of $\cross$'s minus the number of $\bigo$'s in $x_\lambda(\delta)$.
\end{corollary}

\begin{proof} Every bipartition can be constricted by adding a finite number of boxes to $\biemp$, and by the previous proposition the process of adding boxes to a bipartition does not change the number of $\cross$'s minus the number of $\bigo$'s in the corresponding weight diagram.  The result now follows from (\ref{emptywt}).  
\end{proof}

\subsection{Cap diagrams and the numbers $D_{\lambda,\mu}(\delta)$}  Given $\delta\in\K$ and a bipartition $\lambda$, the \emph{cap diagram} $c_\lambda(\delta)$ is constructed as follows:

\begin{enumerate}{\setlength\itemindent{0in} \item[\emph{Step 0:}]  Start with $x_\lambda(\delta)$.

\item[\emph{Step $n$:}] Draw a cap connecting vertices $i$ and $j$ on the number line whenever (i) $i<j$; (ii) $i$ is labelled by $\down$ and $j$ is labelled by $\up$ in $x_\lambda(\delta)$; and (iii) each integer between $i$ and $j$ in $x_\lambda(\delta)$ is either labelled by $\bigo$, labelled by $\cross$, or already part of a cap from an earlier step.}
\end{enumerate}
This process will terminate in a finite number of steps, leaving us with the cap diagram $c_\lambda(\delta)$.  For an example see \cite[Example 6.3.2]{CW}.  

We say that two vertices in $x_\lambda(\delta)$ are \emph{connected} if there is a cap connecting them in $c_\lambda(\delta)$.  Next, given two bipartitions $\lambda$ and $\mu$, we say \emph{$x_\mu(\delta)$ is linked to $x_\lambda(\delta)$} if $x_\mu(\delta)$ if obtained from $x_\lambda(\delta)$ by interchanging the labels on finitely many pairs of connected vertices in $x_\lambda(\delta)$.  Finally, set $$D_{\lambda,\mu}(\delta)=\left\{\begin{array}{ll} 1, & \text{if }x_\mu(\delta)\text{ is linked to }x_\lambda(\delta);\\ 0, & \text{otherwise}.\end{array}\right.$$

\begin{remark}\label{Dint} If $\delta\not\in\Z$ and $\lambda$ is any bipartition, then none of the vertices in $x_\lambda(\delta)$ are labelled $\down$, hence there are no caps in $c_\lambda(\delta)$, hence $D_{\lambda,\mu}(\delta)=\left\{\begin{array}{ll} 1, & \lambda=\mu;\\ 0, & \lambda\not=\mu.\end{array}\right.$
\end{remark}

The following propositions concerning the numbers $D_{\lambda,\mu}(\delta)$ will be useful later.  

\begin{proposition}\label{Dtriang} $D_{\lambda,\lambda}(\delta)=1$, and for $\mu\not=\lambda$ we have  $D_{\lambda,\mu}(\delta)=0$ unless $\mu\subset\lambda$ and $|\mu|\leq |\lambda|-2$.
\end{proposition}

\begin{proof} Assume $\mu\not=\lambda$ is such that $D_{\lambda,\mu}(\delta)=1$.  Let $k$ denote the number of pairs of connected vertices in $x_\lambda(\delta)$ whose labels differ from those in $x_\mu(\delta)$.  Then $I_\up(\mu)$ (resp.~$I_\down(\mu,\delta)$) is obtained from $I_\up(\lambda)$ (resp.~$I_\down(\lambda,\delta)$) by replacing $k$ numbers $a_1,\ldots,a_k$ with numbers $b_1,\ldots,b_k$ such that $b_i<a_i$ (resp.~$b_i>a_i$) for all $1\leq i\leq k$.  In particular, this implies $\mu^\black\subsetneqq\lambda^\black$ (resp.~$\mu^\white\subsetneqq\lambda^\white$).  The result follows.
\end{proof}

\begin{proposition}\label{addrem}
Fix $\delta\in\K$ and a bipartition $\lambda$.

(i) If for all $\mu\in\Rem^\black(\lambda)$ there exists $\nu\in\Add^\white(\lambda)$ with $D_{\nu,\mu}(\delta)=1$, then $x_\lambda(\delta)$ has no adjacent labels of the form $\down\up$, $\bigo\up$, or $\down\cross$.

(ii) If for all $\mu\in\Rem^\white(\lambda)$ there exists $\nu\in\Add^\black(\lambda)$ with $D_{\nu,\mu}(\delta)=1$, then $x_\lambda(\delta)$ has no adjacent labels of the form  $\down\up$, $\down\bigo$, or $\cross\up$.  
\end{proposition}

\begin{proof} (i) Suppose $x_\lambda(\delta)$ has adjacent labels of the form $\bigo\up$, $\down\cross$, or $\down\up$.  Then by Proposition \ref{addremwt}(i) there exists $\mu\in\Rem^\black(\lambda)$ such that $x_\mu(\delta)$ has either a $\bigo$ at a vertex where $x_\lambda(\delta)$ has a $\up$, or  a $\cross$ where  $x_\lambda(\delta)$ has a $\down$.  Now suppose $\nu$ is a bipartition with $D_{\nu,\mu}(\delta)=1$.  Then $x_\nu(\delta)$ must also have either a $\bigo$ at a vertex where $x_\lambda(\delta)$ has a $\up$, or  a $\cross$ where  $x_\lambda(\delta)$ has a $\down$.  Thus, by Proposition \ref{addremwt}(ii), $\nu\not\in\Add^\white(\lambda)$.  Statement (i) follows.  The proof of (ii) is similar.
\end{proof}

\subsection{Almost $(m|n)$-cross bipartitions}\label{almost}  Fix integers $m,n\geq0$.  Following \cite{CW} we call a bipartition $\lambda$  \emph{$(m|n)$-cross}\footnote{See \cite[\S8.7]{CW} for a pictorial description of $(m|n)$-cross bipartitions which justifies its name.} if there exists $k$ with $0\leq k\leq m$ such that $\lambda^\black_{k+1}+\lambda^\white_{m-k+1}\leq n$.  We call $\lambda$ \emph{almost $(m|n)$-cross} if it is not $(m|n)$-cross, but any bipartition strictly contained in $\lambda$ is $(m|n)$-cross.  It is easy to show that $\lambda$ is almost $(m|n)$-cross if and only if $\lambda^\black_{m+2}=\lambda^\white_{m+2}=0$ and $\lambda^\black_{k+1}+\lambda^\white_{m-k+1}=n+1$ whenever $0\leq k\leq m$.  From this perspective we see that $\lambda$ is almost $(m|n)$-cross if and only if the Young diagrams of $\lambda^\black$ and $\lambda^\white$ can be glued together to form an $(m+1)\times(n+1)$ rectangle as in the following picture:
\begin{equation}\label{rectangle}\includegraphics{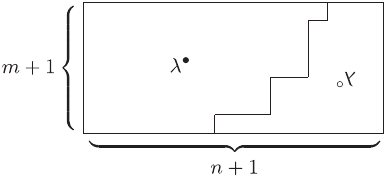}\end{equation}
The following proposition lists a couple useful properties of almost $(m|n)$-cross bipartitions.

\begin{proposition}\label{alprop}  Suppose $\lambda$ is an almost $(m|n)$-cross bipartition.

(i) $|\lambda|=(m+1)(n+1)$.  In particular, all almost $(m|n)$-cross bipartition have the same size.

(ii) For any integers $m',n'\geq0$ with $m'\leq m$ and $n'\leq n$ there exists an almost $(m'|n')$-cross bipartition $\mu$ such that $\mu\subset\lambda$. 
\end{proposition}

\begin{proof} (i) is immediate from (\ref{rectangle}).  For (ii), remove the top $m-m'$ rows and the leftmost $n-n'$ columns of (\ref{rectangle});  the bipartition which corresponds to the remaining $(m'+1)\times(n'+1)$ rectangle is one such $\mu$.\footnote{In fact, one could remove any $m-m'$ rows and $n-n'$ columns.}
\end{proof}

Next, we determine the weight diagrams of almost $(m|n)$-cross bipartitions.

\begin{proposition}\label{alwt} A bipartition $\lambda$ is almost $(m|n)$-cross if and only if its weight diagram is of the form 
\begin{equation*}\label{almostwt}\includegraphics{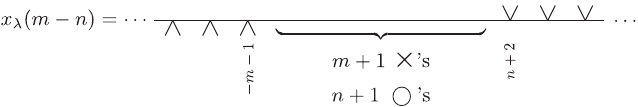}\end{equation*}
\end{proposition}

\begin{proof}  The condition $\lambda^\black_{k+1}+\lambda^\white_{m-k+1}=n+1$ whenever $0\leq k\leq m$ is equivalent to $i-(m-n)-\lambda^\white_i=\lambda^\black_{m-i+2}-(m-i+1)$ whenever $1\leq i\leq m+1$.  Hence, $\lambda$ is almost $(m|n)$-cross if and only if $I_\up(\lambda)=\{\lambda^\black_1,\lambda^\black_2-1,\ldots,\lambda^\black_{m+1}-m, -m-1, -m-2,\ldots\}$ and  $I_\down(\lambda,m-n)  =\{\lambda^\black_{m+1}-m,\ldots,\lambda^\black_2-1,\lambda^\black_1, n+2, n+3,\ldots\}$, which is equivalent to $x_\lambda(m-n)$ having the desired form.
\end{proof}

\begin{corollary}\label{linkage} Suppose $\lambda$ and $\mu$ are bipartitions.

(i) If $\lambda$ is almost $(m|n)$-cross, then $D_{\lambda,\mu}(m-n)=0$ whenever $\mu\not=\lambda$.

(ii) If $\mu$ is almost $(m|n)$-cross, then $D_{\lambda,\mu}(m-n)=0$ unless $\lambda=\mu$ or $|\lambda|>|\mu|+2$.

\end{corollary}

\begin{proof} If $\lambda$ is almost $(m|n)$-cross, then by Proposition \ref{alwt} $c_\lambda(m-n)$ has no caps which implies (i).  To prove (ii), assume that $\mu$ is almost $(m|n)$-cross and $D_{\lambda,\mu}(m-n)=1$.  Since $\mu$ is linked to $\lambda$, the $\cross$'s and $\bigo$'s in $x_\lambda(m-n)$ are in the same positions as those in $x_\mu(m-n)$.  Hence, by Proposition \ref{alwt} applied to $\mu$, there are no adjacent connected vertices in $x_\lambda(m-n)$.  It follows that $|\lambda|>|\mu|+2$ (see for instance \cite[Proposition 6.3.6]{CW}).
\end{proof}

\subsection{One-box-bumps}\label{OBB}  
We close this section by describing a sequence of moves which transforms any given almost $(m|n)$-cross bipartition into any other almost $(m|n)$-cross bipartition.  To start, fix an almost $(m|n)$-cross bipartition $\lambda$.  Notice that if $\lambda^\black$ (resp.~$\lambda^\white$) has a removable box in its $i$th row, then $\lambda^\white$ (resp.~$\lambda^\black$) has an addable box in the $(m+2-i)$th row (see (\ref{rectangle})).  The procedure of removing a box from the $i$th row of $\lambda^\black$ (resp.~$\lambda^\white$) and adding a box to the $(m+2-i)$th row of $\lambda^\white$ (resp.~$\lambda^\black$) will be called a  \emph{rightward (resp.~leftward) one-box-bump}.  From (\ref{rectangle}) it is apparent that performing a one-box-bump to an almost $(m|n)$-bipartition will yield another almost $(m|n)$-cross bipartition.  Moreover, it is easy to see that we can get from any given almost $(m|n)$-cross bipartition to any other almost $(m|n)$-cross bipartition by performing a finite sequence of of one-box-bumps.  For example, the bipartitions  $\lambda=((3,3,1,0,\ldots),(2,0,\ldots))$ and $\mu=((2,2,2,0,\ldots),(1,1,1,0,\ldots))$ are both almost $(2|2)$-cross, and we can get from $\lambda$ to $\mu$ using 2 rightward one-box-bumps followed by an leftward one-box-bump as illustrated using the rectangle description of almost $(m|n)$-cross bipartitions below:
$$\includegraphics{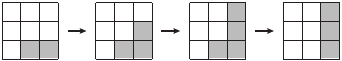}$$
Finally, although it is unnecessary for the rest of this paper, we note that a one-box-bump on an almost $(m|n)$-cross bipartition modifies its weight diagram by swapping a $\cross$ with an adjacent $\bigo$.

\section{Decomposing tensor products in  $\uRep(GL_\delta)$}\label{decompo}

As in \S\ref{uniprop},  given $\delta\in\K$ we let $\uRep(GL_\delta)=\uRep(GL_\delta; \K)$ denote  the $\K$-linear, additive, Karoubian,  symmetric, monoidal category with parameter $\delta$ defined in \cite{Del96} and \cite[\S10]{Del07}\footnote{Deligne uses the notation $\Rep(GL(\delta),\K)$ for $\uRep(GL_\delta)$.  We will follow the notation in \cite{CW}.} (see also \cite[\S3]{CW}).  
In this section we will prove some technical results on decomposing tensor product in $\uRep(GL_\delta)$.  As we will see, while studying $\uRep(GL_\delta)$ it is sometimes useful to work with a ``generic" parameter.  To do so we let $t$ denote an indeterminate and write $\K(t)$ for the field of rational functions in $t$.  Then $\uRep(GL_t)=\uRep(GL_t; \K(t))$ is a $\K(t)$-linear symmetric monoidal category with parameter $t$.  We begin with a brief account of the classification of indecomposable objects in $\uRep(GL_\delta)$ found in \cite{CW}.

\subsection{Indecomposable objects}\label{inde}  Given a bipartition $\lambda=(\lambda^\black,\lambda^\white)$, in \cite[\S4]{CW} an indecomposable object $L(\lambda)=L(\lambda^\black,\lambda^\white)$  in $\uRep(GL_\delta)$ is defined (up to isomorphism).  For example, $L\biemp$ denotes the unit object {\bf 1}, $L(\Box,\emp)$ is the ``natural" object of dimension $\delta$ mentioned in \S\ref{uniprop}, and $L(\emp,\Box)$ is the dual of $L(\Box,\emp)$.  We will not use an explicit definition of $L(\lambda)$ in this paper, so we refer the interested reader to \emph{loc.~cit.} for more details on $L(\lambda)$.  However, the $L(\lambda)$'s have many useful properties which we will exploit.  For instance, the assignment 
\begin{equation}\label{indclass}\begin{array}{rcl}
\left\{\begin{tabular}{c}bipartitions of\\
arbitrary size\end{tabular}\right\} & \arup{\text{bij.}} & \left\{\begin{tabular}{c}nonzero indecomposable objects in\\
 $\uRep(GL_\delta)$, up to isomorphism\end{tabular}\right\}\\[6pt]
 \lambda & \mapsto & L(\lambda)
 \end{array}
 \end{equation} is a bijection 
 \cite[Theorem 4.6.2]{CW} which we will use implicitly throughout the rest of this paper.  
 It is important to note that while the object $L(\lambda)$ depends on the parameter $\delta$, the fact that (\ref{indclass}) is a bijection does not.  Hence indecomposable objects in $\uRep(GL_\delta)$ are parametrized by bipartitions regardless of the parameter $\delta$.  Our notation $L(\lambda)$ for the indecomposable object labelled by $\lambda$ is perhaps poor since it does not keep track of the parameter $\delta$.  In particular, we will  write $L(\lambda)$ to denote both an object in $\uRep(GL_\delta)$ and one in $\uRep(GL_t)$.   Consequently, we must indicate which category we are working in whenever we write about $L(\lambda)$.

\subsection{Decomposition formulae}\label{decform}

The following decomposition formulae for tensor products in $\uRep(GL_t)$ are found in \cite[Corollary 7.1.2]{CW}:
\begin{equation}\label{leftbox}
L(\Box,\emp)\otimes L(\lambda)=\bigoplus_{\mu\in\Add^\black(\lambda)}L(\mu)\oplus\bigoplus_{\nu\in\Rem^\white(\lambda)}L(\nu),
\end{equation}
\begin{equation}\label{rightbox}
L(\emp,\Box)\otimes L(\lambda)=\bigoplus_{\mu\in\Add^\white(\lambda)}L(\mu)\oplus\bigoplus_{\nu\in\Rem^\black(\lambda)}L(\nu).
\end{equation}

It is important to note that (\ref{leftbox}) and (\ref{rightbox}) hold in $\uRep(GL_t)$ rather than in $\uRep(GL_\delta)$.  It turns out that the decomposition of tensor product in the latter category depends on the numbers $D_{\lambda,\mu}(\delta)$ \cite[\S7.2]{CW}.  To make this dependency precise, given a bipartition $\lambda$ and $\delta\in\K$, write $M_\delta(\lambda)=\bigoplus_\mu L(\mu)^{\oplus D_{\lambda,\mu}(\delta)}$ in $\uRep(GL_t)$.  As an immediate consequence of Proposition \ref{Dtriang} we can write \begin{equation}\label{Mdec}M_\delta(\lambda)=L(\lambda)\oplus \bigoplus_{\substack{\mu\subsetneqq\lambda,\\|\mu|\leq|\lambda|-2}} L(\mu)^{\oplus D_{\lambda,\mu}(\delta)}.\end{equation}
The objects $M_\delta(\lambda)$ allow us to compare tensor product decompositions in $\uRep(GL_\delta)$ with those in $\uRep(GL_t)$ via the following theorem:

\begin{theorem}\label{lifting}  $L(\lambda)\otimes L(\mu)=L(\nu^{(1)})\oplus\cdots\oplus L(\nu^{(k)})$ in $\uRep(GL_\delta)$ if and only if $M_\delta(\lambda)\otimes M_\delta(\mu)=M_\delta(\nu^{(1)})\oplus\cdots\oplus M_\delta(\nu^{(k)})$ in $\uRep(GL_t)$.  
\end{theorem}

\begin{proof} See \cite[Theorem 6.2.3(1) and Corollary 6.4.2]{CW}.
\end{proof}

Theorem \ref{lifting} along with formulae (\ref{leftbox}), (\ref{rightbox}), and (\ref{Mdec}) will serve as our main tools for proving the remaining results in \S\ref{decompo}.  Indeed, the proof of the following useful proposition will make use of each of these tools.

\begin{proposition}\label{Reductionpropn} Suppose $\lambda$ and $\mu$ are bipartitions and $\delta\in\K$.  If $\mu\subset\lambda$ then $L(\lambda)$ is a summand of $A\otimes L(\mu)$ for some object $A$ in $\uRep(GL_\delta)$.
\end{proposition}

\begin{proof} If $\mu\subset\nu\subset\lambda$, $L(\mu)$ is a summand of $A\otimes L(\nu)$, and $L(\nu)$ is a summand of $B\otimes L(\lambda)$; then $L(\mu)$ is a summand of $A\otimes B\otimes L(\lambda)$.  Hence, it suffices to consider the case $\lambda\in\Add^\black(\mu)\cup\Add^\white(\mu)$.  In this case, $\lambda$ is a bipartition of maximal size such that $L(\lambda)$ is a summand of either $L(\Box,\emp)\otimes L(\mu)$ or $L(\emp,\Box)\otimes L(\mu)$ in $\uRep(GL_t)$ (formulae (\ref{leftbox}) and (\ref{rightbox})).  Hence $\lambda$ is a maximally sized bipartition such that $L(\lambda)$ is a summand of either $M_\delta(\Box,\emp)\otimes M_\delta(\mu)$ or $M_\delta(\emp,\Box)\otimes M_\delta(\mu)$ in $\uRep(GL_t)$ (formulae (\ref{leftbox}), (\ref{rightbox}), and  (\ref{Mdec})).  It follows that $L(\lambda)$ is a summand of either $L(\Box,\emp)\otimes L(\mu)$ or $L(\emp,\Box)\otimes L(\mu)$ in $\uRep(GL_\delta)$ (Theorem  \ref{lifting} and formula (\ref{Mdec})).
\end{proof}

\subsection{Tensor products and almost $(m|n)$-cross bipartitions}  In this subsection we combine the combinatorial properties of almost $(m|n)$-cross bipartitions (\S\ref{almost}-\ref{OBB}) with our decomposition formulae for tensor products in $\uRep(GL_{m-n})$ (\S\ref{decform}) to derive some technical results on tensor products of indecomposable objects labelled by almost $(m|n)$-cross bipartitions.  These results will be used in \S\ref{mainresult} to prove our classification of ideals.  

\begin{proposition}\label{lowerlimit} Fix a bipartion $\lambda\not=\biemp$.  Given another bipartition $\mu$, let $a_\mu$ (resp.~$b_\mu$) denote the number of times $L(\mu)$ occurs in a decomposition of $L(\Box,\emp)\otimes L(\lambda)$ (resp.~$L(\emp,\Box)\otimes L(\lambda)$) into indecomposable summands in the category $\uRep(GL_\delta)$.  If $a_\mu=b_\mu=0$ whenever $\mu\subsetneqq\lambda$, then $\lambda$ is almost $(m|n)$-cross for some $m,n\in\Z_{\geq0}$ with $m-n=\delta$.  
\end{proposition}

\begin{proof}  Write $L(\Box,\emp)\otimes L(\lambda)=L(\nu^{(1)})\oplus\cdots\oplus L(\nu^{(k)})$ in $\uRep(GL_\delta)$ and assume $\nu^{(i)}$ is not strictly contained in $\lambda$ for all $i=1,\ldots,k$.  Then, by Theorem  \ref{lifting}, $M_\delta(\Box,\emp)\otimes M_\delta(\lambda)=M_\delta(\nu^{(1)})\oplus\cdots\oplus M_\delta(\nu^{(k)})$ in $\uRep(GL_t)$. Hence, by formulae (\ref{Mdec}) and  (\ref{leftbox}), \begin{equation}\label{nubound}|\nu^{(i)}|\leq|\lambda|+1\end{equation} for each $i=1,\ldots,k$.  Moreover, equality in (\ref{nubound}) occurs if and only if $\nu^{(i)}\in\Add^\black(\lambda)$.

Now, given $\mu\in\Rem^\white(\lambda)$, $L(\mu)$ is a summand of $L(\Box,\emp)\otimes L(\lambda)$ in $\uRep(GL_t)$ (formula (\ref{leftbox})), and is thus a summand of $M_\delta(\Box,\emp)\otimes M_\delta(\lambda)$ in $\uRep(GL_t)$ (formula (\ref{Mdec})).  Hence, $D_{\nu^{(j)},\mu}(\delta)\not=0$ for some $j=1,\ldots,k$. However, $a_\mu=0$ implies $\mu\not=\nu^{(j)}$, which implies $\delta\in\Z$ (Remark \ref{Dint}).  Moreover, $|\nu^{(j)}|\geq|\mu|+2=|\lambda|+1$ (Proposition \ref{Dtriang}), which implies $\nu^{(j)}\in\Add^\black(\lambda)$.  In particular, we have shown that for every $\mu\in\Rem^\white(\lambda)$ there exists $\nu\in\Add^\black(\lambda)$ with $D_{\nu,\mu}(\delta)\not=0$.  Similarly, assuming $b_\mu=0$ for all $\mu\subsetneqq\lambda$, one can show that for every $\mu\in\Rem^\black(\lambda)$ there exists $\nu\in\Add^\white(\lambda)$ with $D_{\nu,\mu}(\delta)\not=0$.  It follows from Proposition \ref{addrem} that the label to the left (resp.~right) of any $\up$ (resp.~$\down$) in $x_\lambda(\delta)$ is also a $\up$ (resp.~$\down$).  Thus, $x_\lambda(\delta)$ must be of the form $$\includegraphics{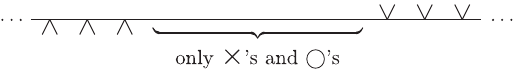}$$  Let $m\in\Z$ be such that the number of $\cross$'s in $x_\lambda(\delta)$ is $m+1$.  Note that $m\geq0$ since $\lambda\not=\biemp$.  Moreover, there are $n+1$ $\bigo$'s in $x_\lambda(\delta)$ where $n:=m-\delta$ (Corollary \ref{xminuso}).  It follows from Proposition \ref{alwt} that $\lambda$ is almost $(m|n)$-cross.
\end{proof}

\begin{corollary}\label{toalmost} Suppose $\lambda$ is a bipartition and $\delta\in\K$.  Either

(i) ${\bf 1}$ is a summand of $A\otimes L(\lambda)$ for some object $A$ in $\uRep(GL_\delta)$; or 

(ii) $\delta\in\Z$ and there exists a nonnegative integer $m$ and an almost $(m|m-\delta)$-cross bipartition $\mu\subset\lambda$ such that $L(\mu)$ is a summand of $A\otimes L(\lambda)$ for some object $A$ in $\uRep(GL_\delta)$.
\end{corollary}

\begin{proof}  We induct on $|\lambda|$.  If $\lambda=\biemp$ or $\lambda$ is almost $(m|m-\delta)$-cross for some $m$, then take $A={\bf 1}$.  Otherwise, by Proposition \ref{lowerlimit} there exists a bipartition $\nu$ with $\nu\subsetneqq\lambda$ such that $L(\nu)$ is a summand of either $L(\Box,\emp)\otimes L(\lambda)$ or $L(\emp,\Box)\otimes L(\lambda)$ in $\uRep(GL_\delta)$.  By induction, either (i) ${\bf 1}$ is a summand of $A'\otimes L(\nu)$ and hence a summand of either $A'\otimes L(\Box,\emp)\otimes L(\lambda)$ or  $A'\otimes L(\emp,\Box)\otimes L(\lambda)$ for some object $A'$; or (ii) $\delta\in\Z$ and there exists a nonnegative integer $m$ and an almost $(m|m-\delta)$-cross bipartition $\mu\subset\lambda$ such that $L(\mu)$ is a summand of $A'\otimes L(\nu)$ and hence of either $A'\otimes L(\Box,\emp)\otimes L(\lambda)$ or  $A'\otimes L(\emp,\Box)\otimes L(\lambda)$ for some object $A'$.  
\end{proof}

\begin{lemma}\label{redgen} Suppose $\lambda$ and $\mu$ are almost  $(m|n)$-cross  bipartitions. 

(i)  $L(\mu)$ is a summand of $L(\Box,\emp)^{\otimes 2}\otimes L(\lambda)$  in  $\uRep(GL_{m-n})$ if and only if $L(\mu)$ is a summand of $L(\Box,\emp)^{\otimes 2}\otimes L(\lambda)$ in  $\uRep(GL_{t})$.

(ii) $L(\mu)$ is a summand of  $L(\emp,\Box)^{\otimes 2}\otimes L(\lambda)$ in $\uRep(GL_{m-n})$ if and only if $L(\mu)$ is a summand of $L(\emp,\Box)^{\otimes 2}\otimes L(\lambda)$ in  $\uRep(GL_{t})$.
\end{lemma}

\begin{proof}  (i) Set $\delta=m-n$ and write 
$L(\Box,\emp)^{\otimes 2}\otimes L(\lambda)=L(\mu^{(1)})\oplus\cdots\oplus L(\mu^{(k)})$
in $\uRep(GL_\delta)$.  Since $M_\delta(\Box,\emp)=L(\Box,\emp)$ and $M_\delta(\lambda)=L(\lambda)$ (formula (\ref{Mdec}) and Corollary \ref{linkage}(i)), it follows that  \begin{equation}\label{ms} L(\Box,\emp)^{\otimes 2}\otimes L(\lambda)=M_\delta(\mu^{(1)})\oplus\cdots\oplus M_\delta(\mu^{(k)})\end{equation}
in $\uRep(GL_t)$ (Theorem  \ref{lifting}).  In particular, $|\mu^{(i)}|\leq |\lambda|+2=|\mu|+2$ for all $i=1,\ldots,k$ (formulae (\ref{leftbox}), (\ref{Mdec}), and Proposition \ref{alprop}(i)).    Now, $L(\mu)$ is a summand of (\ref{ms}) in $\uRep(GL_t)$ if and only if $D_{\mu^{(i)},\mu}(\delta)\not=0$ for some $i$, which occurs if and only if $\mu=\mu^{(i)}$ for some $i$ (Corollary \ref{linkage}(ii)).  The proof of (ii) is similar.
\end{proof}

\begin{proposition}\label{almosttoalmost} If $\lambda$ and $\mu$ are both almost $(m|n)$-cross, then there exist almost $(m|n)$-cross bipartitions $\nu^{(0)},\ldots,\nu^{(k)}$ such that $\nu^{(0)}=\mu$, $\nu^{(k)}=\lambda$, and  $L(\nu^{(i-1)})$ is a summand of either $L(\Box,\emp)^{\otimes 2}\otimes L(\nu^{(i)})$ or $L(\emp,\Box)^{\otimes 2}\otimes L(\nu^{(i)})$ in $\uRep(GL_{m-n})$ for each $i=1,\ldots,k$.
\end{proposition}

\begin{proof} Recall from \S\ref{OBB} that $\lambda$ can be obtained from $\mu$ by a finite sequence of one-box-bumps.  Let $k$ denote the number of one-box-bumps and let $\nu^{(i)}$ denote the almost $(m|n)$-cross bipartition arising after the $i$th one-box-bump for $i=0,\ldots,k$.   By Lemma \ref{redgen} we  must show  $L(\nu^{(i-1)})$ is a summand of either $L(\Box,\emp)^{\otimes 2}\otimes L(\nu^{(i)})$ or $L(\emp,\Box)^{\otimes 2}\otimes L(\nu^{(i)})$ in $\uRep(GL_{t})$ for all $i=1,\ldots,k$.  If $\nu^{(i)}$ is obtained from $\nu^{(i-1)}$ by a rightward one-box-bump, then $L(\nu^{(i-1)})$ is a summand of $L(\Box,\emp)^{\otimes 2}\otimes L(\nu^{(i)})$ by (\ref{leftbox}).  Similarly, if  $\nu^{(i)}$ is obtained from $\nu^{(i-1)}$ by a leftward one-box-bump, then $L(\nu^{(i-1)})$ is a summand of $L(\emp,\Box)^{\otimes 2}\otimes L(\nu^{(i)})$ by  (\ref{rightbox}).
\end{proof}

\begin{corollary}\label{Bothalmost} If $\lambda$ and $\mu$ are both almost $(m|n)$-cross, then $L(\mu)$ is a summand of $A\otimes L(\lambda)$ for some object $A$ in $\uRep(GL_{m-n})$.
\end{corollary}

\begin{proof} Let $\nu^{(0)},\ldots,\nu^{(k)}$ be as in Proposition \ref{almosttoalmost}.  Inducting on $k$, we can assume $L(\mu)$ is a summand of $A'\otimes L(\nu^{(k-1)})$ for some $A'$.  
If $L(\nu^{(k-1)})$ is a summand of $L(\Box,\emp)^{\otimes 2}\otimes L(\lambda)$, set $A=A'\otimes L(\Box,\emp)^{\otimes 2}$.  Otherwise, set $A=A'\otimes L(\emp,\Box)^{\otimes 2}$.
\end{proof}

\section{Ideals in  $\uRep(GL_\delta)$}\label{mainresult}

In this section we classify the ideals in $\uRep(GL_\delta)$, which is the main result of this paper.  To start, notice that Definition \ref{ideal} implies that ideals are always closed under taking direct summands: $A\oplus B\in\I$ implies $A,B\in\I$.  However, the converse implication does not follow from  Definition \ref{ideal} in an arbitrary braided monoidal category.  We call an ideal $\I$ \emph{additive} if it is closed under taking direct sums:  $A,B\in\I$ implies $A\oplus B\in\I$.  Our first result on ideals of $\uRep(GL_\delta)$ is that additivity is guaranteed.   

\begin{proposition}\label{additive} Ideals in $\uRep(GL_\delta)$ are additive for all $\delta\in\K$.  
\end{proposition}

\begin{proof} Let $\I$ be an ideal in $\uRep(GL_\delta)$ and $A,B\in\I$.  If $A$ and $B$ are zero there is nothing to show, so we may assume $A\oplus B=L(\lambda^{(1)})\oplus\cdots\oplus L(\lambda^{(k)})$ for some bipartitions $\lambda^{(1)},\ldots,\lambda^{(k)}$.  In particular, each $L(\lambda^{(i)})\in\I$ since $L(\lambda^{(i)})$ is a summand of either $A$ or $B$.  Since $\uRep(GL_\delta)$ is additive, we may assume $\I$ is proper so that ${\bf 1}\not\in\I$.  Hence, each $\lambda^{(i)}$  does not satisfy part (i) of Corollary \ref{toalmost}.  
Consider the nonnegative integers $m_1,\ldots,m_k$ and bipartitions $\mu^{(1)},\ldots,\mu^{(k)}$ such that $\mu^{(i)}$ is almost $(m_i|m_i-\delta)$-cross and $\mu^{(i)}\subset\lambda^{(i)}$ for each $i$, prescribed by Corollary \ref{toalmost}(ii).  Let $j$ be such that $m_j\leq m_i$ whenever $1\leq i\leq k$.  Then each $\lambda^{(i)}$ is not $(m_j|m_j-\delta)$-cross, so there exist almost $(m_j|m_j-\delta)$-cross bipartitions $\nu^{(1)},\ldots,\nu^{(k)}$ with $\nu^{(i)}\subset\lambda^{(i)}$ for each $i$.  
Now, by Corollary \ref{toalmost} there exists an object $C$ such that $L(\mu^{(j)})$ is a summand of $C\otimes L(\lambda^{(j)})$.  By Corollary \ref{Bothalmost} there exist  objects $C_1,\ldots,C_k$ such that $L(\nu^{(i)})$ is a summand of $C_i\otimes L(\mu^{(j)})$ for each $i$. By Proposition \ref{Reductionpropn} there exist objects $C'_1,\ldots,C'_k$ such that $L(\lambda^{(i)})$ is a summand of $C'_i\otimes L(\nu^{(i)})$ for each $i$.  It follows that $A\oplus B$ is a summand of $\left( \bigoplus_i C'_i\right)\otimes\left(\bigoplus_i C_i\right)\otimes C\otimes L(\lambda^{(j)})\in\I$. 
\end{proof}



\subsection{The ideals $\cat{I}(m|n)$}  Fix $m,n\in\Z_{\geq0}$ and let $\I(m|n)$ be as in \S\ref{uniprop}.
The indecomposable objects in $\I(m|n)$ have been worked out in \cite{CW}:


\begin{theorem}\label{vanish} \cite[Theorem 8.7.6]{CW} $L(\lambda)\in\I(m|n)$ if and only if $\lambda$ is not $(m|n)$-cross.
\end{theorem}

The theorem above allows us to prove the following lemma, which will be crucial in the proof of the classification of ideals in $\uRep(GL_\delta)$.

\begin{lemma}\label{ImninI} Suppose $\I$ is an ideal in $\uRep(GL_{m-n})$ and $\lambda$ is an almost $(m|n)$-cross bipartition.  If $L(\lambda)\in\I$, then $\I(m|n)\subset\I$.
\end{lemma}

\begin{proof}  If $L(\lambda)\in\I$, then by Corollary \ref{Bothalmost} $L(\mu)\in\I$ for every almost $(m|n)$-cross bipartition $\mu$.  Since every non-$(m|n)$-cross bipartition contains an almost $(m|n)$-bipartition, by Proposition \ref{Reductionpropn} $L(\mu)\in\I$ for every non-$(m|n)$-cross bipartition $\mu$.  Hence we are done by Theorem \ref{vanish} and Proposition \ref{additive}.
\end{proof}

\subsection{Classification of ideals}\label{main result}  We can now prove the main result of this paper:

\begin{theorem}\label{Classification} The set $\{\I(m|n) ~|~m,n\in\Z_{\geq0}, m-n=\delta\}$ is a complete set of pairwise distinct nontrivial proper ideals in $\uRep(GL_\delta)$.
\end{theorem}

\begin{proof}  
It follows from Theorem \ref{vanish} that the ideals $\I(m|n)$ are pairwise distinct, nontrivial, and proper.  
Suppose $\I$ is a nontrivial proper ideal in $\uRep(GL_\delta)$ and let $\lambda$ be a bipartition which is minimal with respect to inclusion of bipartitions such that $L(\lambda)\in\I$.  Since $\I$ is proper, $\lambda\not=\biemp$.  Hence, by Proposition \ref{lowerlimit} and the minimality of $\lambda$, there exist $m,n\geq0$ with $m-n=\delta$ such that $\lambda$ is almost $(m|n)$-cross.  It follows from Lemma \ref{ImninI} that $\I(m|n)\subset\I$.

To complete the proof of the theorem, we show $\I\subset\I(m|n)$.  To do so, since $\I(m|n)$ is additive, it suffices to show $L(\mu)\in\I(m|n)$ whenever $\mu$ is a bipartition with $L(\mu)\in\I$.  By Proposition \ref{Reductionpropn} we may assume $\mu$ is minimal in $\I$ with respect to containment of bipartitions.  In this case, by Proposition \ref{lowerlimit}  $\mu$ is almost $(m'|n')$-cross for some $m',n'\geq 0$ with $m'-n'=\delta$.  It follows from Lemma \ref{ImninI} that $\I(m'|n')\subset\I$.  Hence, by the minimality of $\lambda$ along with Proposition \ref{alprop}(ii), $m'\geq m$ and $n'\geq n$.  Thus $L(\mu)\in\I(m'|n')\subset\I(m|n)$.  \end{proof}

\nocite{BS,BS1,BS2,BS3,BS4}

\renewcommand{\bibname}{\textsc{references}} 
\bibliographystyle{alphanum}	
	\bibliography{references}

\end{document}